%% file: author.tex
\documentclass[graybox]{svmult}

\usepackage{mathptmx}       
\usepackage{helvet}         
\usepackage{courier}        
\usepackage{type1cm}        
\usepackage{makeidx}         
\usepackage{graphicx}        
\usepackage{multicol}        
\usepackage[bottom]{footmisc}

\usepackage{amsmath}
\usepackage{amssymb}

\makeindex             

\begin{document}

\title*{Synchronization, Lyapunov exponents and stable manifolds for random dynamical systems}
\titlerunning{Synchronization and Lyapunov exponents} 
\author{Michael Scheutzow and Isabell Vorkastner}
\institute{Michael Scheutzow \at Institut f\"ur Mathematik, MA 7-5, Fakult\"at II,
	Technische Universit\"at Berlin, Stra{\ss}e des 17. Juni 136, 10623 Berlin; 
	\email{ms@math.tu-berlin.de}
	\and Isabell Vorkastner \at Institut f\"ur Mathematik, MA 7-5, Fakult\"at II,
	Technische Universit\"at Berlin, Stra{\ss}e des 17. Juni 136, 10623 Berlin; 
	\email{vorkastn@math.tu-berlin.de}}
%
%
\maketitle

\abstract{During the past decades, the question of existence and properties of a random attractor of a random dynamical system generated by an S(P)DE has received 
considerable attention, for example by the work of Gess and R\"ockner. Recently some authors investigated sufficient conditions which guarantee {\em synchronization}, i.e.~existence of a random attractor which is a singleton. It is reasonable to conjecture that synchronization and negativity (or non-positivity) of the top 
Lyapunov exponent of the system should be closely related since both mean that the system is contracting in some sense. Based on classical results by Ruelle, we 
formulate positive results in this direction. Finally we provide two very simple but striking examples of one-dimensional monotone random dynamical systems for which 0 is a fixed point. In the first example, the Lyapunov exponent is strictly negative but nevertheless all trajectories starting outside of 0 
diverge to $\infty$ or $-\infty$. In particular, there is no synchronization (not even locally).  In the second example (which is just the time reversal of the first), the Lyapunov exponent is strictly positive but nevertheless there is synchronization.}

\vspace{10pt} 

\newpage
\input{notation}
\input{stable_manifold}

\input{examples}

\bibliographystyle{spmpsci}
\bibliography{mybib}
\end{document}

%% file: notation.tex
\section{Introduction}

During the past decades, the question of existence and properties of a random attractor of a random dynamical system generated by an S(P)DE has received 
considerable attention, see for example \cite{Crauel1994}, \cite{Gess2011}, \cite{Beyn2011}. Recently some papers investigated sufficient conditions 
which guarantee {\em synchronization}, i.e.~existence of a random attractor which is a singleton, see \cite{Chueshov2004}, \cite{Flandoli2016}, 
\cite{Gess2016},  \cite{Cranston2016}, \cite{Vorkastner2016}. It is reasonable to conjecture that synchronization and negativity (or non-positivity) 
of the top Lyapunov exponent of the system should be closely related since both mean that the system is contracting in some sense. 
A positive result of that kind in the finite dimensional case is \cite[Lemma 3.1]{Flandoli2016} which states that (under an ergodicity assumption) 
negativity of the top Lyapunov exponent plus an integrability assumption on the derivative in a neighborhood of the support of the invariant measure
guarantees that for almost every $x$ in the support of the invariant measure, there exists a random neighborhood of $x$ which forms a local stable manifold. 
In particular, the system contracts locally. In the present paper, we formulate a corresponding result for separable Hilbert spaces. 
Like  \cite[Lemma 3.1]{Flandoli2016} the proof is an easy consequence of results by Ruelle \cite{Ruelle82}. Example 1 in Section \ref{examples} shows that 
the result becomes untrue if the integrability assumption on the derivative is dropped. In Example 1 we investigate a simple one-dimensional random 
dynamical system generated by independent and identically distributed strictly monotone and bijective maps from the real line to itself which fix the point 
0. The Lyapunov exponent is strictly negative but nevertheless the point 0 is not even locally asymptotically stable. In fact all trajectories starting 
outside 0  go to $\infty$ or $-\infty$ (depending on the sign of the initial condition). In particular, there is no synchronization.  
The reason for this behaviour is that the random function is very steep outside a very small (random) neighborhood of 0 (even though the derivative at 
0 is $1/2$ almost surely). 

We also consider the opposite behaviour. Proposition \ref{proposition2} requires that the unstable manifold $U$ of a random fixed point is non-trivial and 
states that under this condition, synchronization cannot hold. Example 2 shows that replacing the non-triviality of $U$ by positivity of the top Lyapunov 
exponent does not imply lack of synchronization. In fact, Example 2 is just the time reversal of Example 1.

\section{Preliminaries and notation}

Let $(\mathcal{H},\|\cdot\|)$ be a separable Hilbert space with Borel $\sigma$-algebra $\mathcal{B}(\mathcal{H})$
and let $(\Omega, \mathcal{F}, \mathbb{P}, \theta)$ be a metric dynamical system, i.e.
$(\Omega, \mathcal{F},\mathbb{P})$ is a probability space and $(\theta_t)_{t\in \mathbb{T}}$ a group of
jointly measurable maps on $(\Omega, \mathcal{F},\mathbb{P})$ such that $\theta_0=\mathrm{id}$ with invariant measure $\mathbb{P}$.
Here, $\mathbb{T}$ is either $\mathbb{Z}$ or $\mathbb{R}$. We denote the set of non-negative numbers in $\mathbb{T}$ by  ${\mathbb{T}}_+$.
\\
Further, let $\varphi : \mathbb{T}_+ \times \Omega \times \mathcal{H} \rightarrow \mathcal{H}$ be jointly measurable, $\varphi_0 (\omega,x) =x$, 
$\varphi_{s+t} (\omega,x) = \varphi_t (\theta_s \omega, \varphi_s (\omega,x))$ for all $x \in \mathcal{H}$,
and  $ x \mapsto \varphi_t (\omega, x)$ continuous,
$s,t \in \mathbb{T}_+$ and $\omega \in \Omega$. The collection 
$(\Omega, \mathcal{F}, \mathbb{P}, \theta,\varphi)$ is then called a \emph{random dynamical system},
see \cite{Arnold1988} for a comprehensive treatment.
\\
We suppose that there is a family $\left( \mathcal{F}_{s,t}\right)_{-\infty \leq s \leq t \leq \infty}$
of sub-$\sigma$ algebras of $\mathcal{F}$ such that
$\theta^{-1}_r \left( \mathcal{F}_{s,t} \right) = \mathcal{F}_{s+r,t+r}$ for all $r,s,t$,  
$\mathcal{F}_{t,u} \subset \mathcal{F}_{s,v}$ whenever $s \leq t \leq u \leq v$, and $\varphi_t (\cdot, x)$ is $\mathcal{F}_{0,t}$-measurable for each 
$ t \in \mathbb{T}_+$. 
If, additionally, $\mathcal{F}_{s,t}$ and $\mathcal{F}_{u,v}$ are independent whenever $s \leq t \leq u \leq v$, 
then  the collection $(\Omega, \mathcal{F}, \mathbb{P}, \theta,\varphi)$ is called a
\emph{white-noise random dynamical system}. We will generally not assume the white-noise property in the following. 
\\
Define the \emph{skew product} $\Theta$ on $\Omega \times \mathcal{H}$ by 
$\Theta_t (\omega,x):= ( \theta_t \omega, \varphi_t (\omega,x))$ for $t \in \mathbb{T}_+$.
We then say that a probability measure $\rho$ on $\Omega \times \mathcal{H}$ is an \emph{invariant measure} for a 
random dynamical system if its marginal on $\Omega$ is $\mathbb{P}$ and $\Theta (t) \rho = \rho$
for all $t \in \mathbb{T}_+$. 
\begin{definition}
	A family $\left\{ A (\omega ) \right\}_{\omega \in \Omega}$ of non-empty subsets of $\mathcal{H}$
	is called 
	\begin{enumerate}
	\item[(i)] 
		a \emph{random point} (resp. \emph{random compact set}) if it is $\mathbb{P}$-almost surely a point
		(resp. compact set) and 
		$\omega \mapsto \inf_{a \in A(\omega)} \left\| x - a \right\|$ is $\mathcal{F}$-measurable 
		for each $x \in \mathcal{H}$.
	\item[(ii)]
		\emph{$\varphi$-invariant} if $\varphi_t (\omega, A(\omega)) = A(\theta_t \omega )$ for all 
		$t \in \mathbb{T}_+$ and almost all $\omega \in \Omega$.
	\end{enumerate}
\end{definition}
\begin{definition}
	Let $(\Omega, \mathcal{F}, \mathbb{P}, \theta, \varphi)$ be a random dynamical system. 
	A random compact set $A$ is called a \emph{weak attractor} if 
	it satisfies the following properties
	\begin{enumerate}
	\item [(i)] $A$ is $\varphi$-invariant
	\item [(ii)] for every compact set $B \subset \mathcal{H}$ 
		\begin{align*}
			\lim_{t \rightarrow \infty} \sup_{x \in B} \inf_{a \in A(\omega)}
				\left\| \varphi_t (\theta_{-t}\omega ,x) - a \right\| =0 
				\qquad  \textrm{in probability}.
		\end{align*}
	\end{enumerate}
If the convergence in (ii) is even almost sure, then $A$ is called a  \emph{pullback attractor}.
\end{definition}

\begin{definition}
        Let $(\Omega, \mathcal{F}, \mathbb{P}, \theta, \varphi)$ be a random dynamical system.
	\emph{Synchronization} occurs if there is a weak attractor $A(\omega)$ being 
	a singleton for $\mathbb{P}$-almost every $\omega \in \Omega$.
\end{definition}

Further, we use the following notation. Denote by $B(x,r)$
the open ball centered at $x \in \mathcal{H}$ with radius $r >0$ and by $\bar{B}(x,r)$ the respective 
closed ball. We say  that a map is of class $C^{1,\delta}$ if its derivative is H\"older continuous
of exponent $\delta$. Denote by $D$ the derivative in the state space. 

%% file: stable_manifold.tex
\section{Top Lyapunov exponent and synchronization}

In this section we demonstrate some relations between the sign of the top Lyapunov exponent, stable/unstable
submanifolds and synchronization of a random dynamical system $\varphi$.

\begin{proposition}
	\label{stable_manifold}
	Let $\varphi$ be a discrete or continuous time random dynamical system with state space $\mathcal{H}$ and 
        assume that $\varphi_1 (\omega, \cdot ) \in C^{1, \delta} $ 
	for some $\delta \in (0,1)$. 
	Assume further that $\varphi$ has an invariant measure $\rho$ 
	such that
	\begin{equation*}
		\label{condition1}
		 \int_{\Omega \times \mathcal{H}} 
		 \log^+ \left\| D \varphi_1 (\omega,x ) \right\| \D \rho (\omega,x) < \infty.
	\end{equation*}	
	and
	\begin{equation}
		\label{condtion2}
		\int_{\Omega \times \mathcal{H}} \log^+ \left( 
		\left\| \varphi_1 (\omega, \cdot + x ) - \varphi_1 (\omega, x ) \right\|_{C^{1,\delta} 
		(\bar{B}(0,1))}
		\right) \D \rho(\omega,x) < \infty	.	
	\end{equation}
	Then, the (discrete-time) top Lyapunov exponent
	\begin{equation*}
		\lambda (\omega , x) = \lim_{n \rightarrow \infty} \frac{1}{n} \log 
		\left\| D \varphi_n (\omega,x ) \right\|
	\end{equation*}
	is defined for $\rho$-almost all $(\omega,x) \in \Omega \times \mathcal{H}$.
	Assume that there exists some $\mu <0$ such that $ \lambda (\omega ,x) < \mu$ almost everywhere.
	Then, there exist measurable functions $0< \alpha (\omega,x) < \beta (\omega,x) <1$ such that
	for $\rho$-almost all $(\omega,x)$
	\begin{equation*}
		S(\omega,x) = \left\{ y \in \bar{B}(x, \alpha (\omega,x)) : 
		\left\| \varphi_n (\omega, y ) - \varphi_n (\omega, x ) \right\| \leq \beta(\omega,x) \E ^{\mu n}
		\textrm{ for all } n \geq 0 \right\}
	\end{equation*}
	is a measurable neighborhood of $x$. We refer to $S(\omega,x)$ as the stable manifold.
\end{proposition}

\begin{proof}
	By the same construction as in \cite[Lemma 3.1]{Flandoli2016}, define $M:= \Omega \times \mathcal{H}$,
	$\tilde{\mathcal{F}} := \mathcal{F} \otimes \mathcal{B}(\mathcal{H})$ and 
	$f : M \mapsto M$
	given by $f (m) := \Theta_1 (\omega,x)$ for $m = (\omega,x) \in M$.
	Further, set 
	\begin{equation*}
		F_m (y) := \varphi_1 (\omega, y + x ) - \varphi_1 (\omega, x ) \quad \textrm{for } m=(\omega,x) \in M
	\end{equation*}
	and apply \cite[Theorem 5.1]{Ruelle82} with $Q=0$. 
	Observe that the set $D^v$ in the proof of \cite[Theorem 5.1]{Ruelle82}
	is -- in our special case -- both open and closed in the ball $\bar{B}(0, \alpha(\omega))$ and therefore 
	$\bar{B}(0, \alpha (\omega)) \subset D^v$. This implies $ \bar{B}(x, \alpha(\omega,x)) \subset S(\omega,x)$
	almost everywhere and therefore $ \bar{B}(x, \alpha(\omega,x)) = S(\omega,x)$ almost surely.
	\smartqed \qed
\end{proof}

\begin{corollary}\label{coro}
        Under the assumptions of the previous proposition the random dynamical system $\varphi$ is asymptotically stable, i.e.~there 
        exists  a deterministic non-empty, open set $U$ in $\mathcal{H}$ such that
	\begin{equation*}
		\mathbb{P} \left( \lim_{n \rightarrow \infty} \mathrm{diam} 
		\left( \varphi_n ( \cdot, U )
		\right) =0 \right) >0.
	\end{equation*}
\end{corollary}

\begin{proof}
	Existence of a neighborhood $B(\omega, \alpha(\omega,x))$ as in Proposition \ref{stable_manifold} implies local asymptotic
	stability by \cite[Lemma 3.3]{Flandoli2016} (the latter Lemma is formulated and proved only in the finite-dimensional 
	case but the proof in our set-up is almost identical). 
    \smartqed \qed
\end{proof}

	If we further assume that $\varphi$ is a  white-noise random dynamical system, $\varphi$ has a weak attractor, 
	satisfies an irreducibility
	condition and  contracts on large sets, then synchronization follows,
	see \cite[Theorem 2.14]{Flandoli2016} for exact conditions.\\

\par

Random attractors with positive top Lyapunov exponent are more difficult to characterize.

\begin{proposition}
	\label{proposition2}
	Assume there exist a $\varphi$-invariant random point $A(\omega)$, some $\mu > 0$ and
	some measurable functions $0< \alpha (\omega) < \beta (\omega)<1$
	such that
	\begin{eqnarray}
			\label{unstable_submanifold}
			\begin{split}
			U(\omega) = \big\{ x_0 \in \bar{B} (A(\omega), \alpha (\omega)):
			\exists (x_n)_{n \in \mathbb{N}} \textrm{ with } \varphi (\theta_{-n} \omega, x_n) = x_{n-1} \\
			\textrm{ and}
			\left\| x_n - A(\theta_{-n} \omega) \right\| \leq \beta (\omega) \E ^{-\mu n}
			\textrm{ for all } n \geq 0
			\big\}
			\end{split}
		\end{eqnarray}
	is not trivial (i.e.~consists of more than one point) almost surely.
	Further assume there exists some $x_0 (\omega) \in U(\omega) \setminus A(\omega)$
	such that  $x_n(\omega)$ are random points for
	$n \geq 0$ where $x_n (\omega)$ are chosen as in (\ref{unstable_submanifold}).
	Then, the random dynamical system $\varphi$ does not synchronize.
\end{proposition}

\begin{proof}
	Suppose $\varphi$ does synchronize. Then, there exists a weak attractor $\tilde{A}$ being
	a single random point. By the same arguments as in \cite[Lemma 1.3]{Flandoli2016} (stating uniqueness of a weak attractor), 
        $A$ and $\tilde{A}$ have to agree almost surely.
	Let $(x_n (\omega))_{n \in \mathbb{N}_0}$ be as in the proposition.
	There exists some $q>0$ such that
	\begin{equation*}
		\mathbb{P} \left( \left\| A(\omega) - x_0(\omega)\right\| > q \right) > \frac{3}{4} \; .
	\end{equation*}
	By \cite[Proposition 2.15]{Crauel2003}, there exists some compact set $K$ such that
	 $ \mathbb{P} \left( A (\omega) \subset K \right) > 3/4 $. 
	 Define the index set $I ( \omega ) = \left\{ n \in \mathbb{N}_0 : A(\theta_{-n} \omega ) \in K \right\}$.
	 Then,
	 \begin{align*}
	 	\mathbb{P} \left( n \in I(\omega) \right) = \mathbb{P} \left( A(\theta_{-n} \omega ) \in K \right)
	 	> \frac{3}{4}
	 \end{align*}
	 for every $n \in \mathbb{N}_0$ 
         and 
         $$
         \lim_{m \to \infty} \big( \inf_{y \in K}\|x_m(\omega)-y\| 1_{\{m \in I(\omega)\}}\big)     =0 
         $$ 
         Therefore, the set
          $$
          \hat K(\omega):= K \cup \{x_m(\omega)\}_{m \in I(\omega)}
          $$
         is a random compact set and hence, by  \cite[Proposition 2.15]{Crauel2003}, 
         there exists a deterministic compact set $\tilde K$ 
         such that $\mathbb{P}(\hat K (\omega) \subset \tilde K)\ge 3/4$.

         Combining these estimates, it follows for each $n \in \mathbb{N}_0$ that
	 \begin{eqnarray*}
	 	&&\mathbb{P} \left( \sup_{y \in \tilde{K}} 
	 		\left\| \varphi_n (\theta_{-n} \omega , y ) - A(\omega) \right\| >q \right) \\
	 	&&\geq \mathbb{P} \left( 
	 		\left\|  \varphi_n (\theta_{-n} \omega , x_n (\omega) ) - A(\omega) \right\| >q, \;
	 		x_n(\omega) \in \tilde K\right) \geq \frac 34 -\mathbb{P} \left( x_n\notin \tilde K\right)\\
                &&\geq  \frac 34 -\mathbb{P} \left( \hat K (\omega) \nsubseteq  \tilde K\right)-  
                \mathbb{P} \left( x_n \notin \hat K(\omega)\right) 
                 \geq  \frac 34 - \frac 14 -\frac 14 =\frac 14.
	 \end{eqnarray*}
         Therefore, there is no synchronization. 
	 \smartqed \qed

\end{proof}

Note that the set $U(\omega)$ is typically the unstable manifold with respect to the invariant measure 
$\rho (\D \omega, \D x) = \delta_{A(\omega)}(\D x) \, \mathbb{P}(\D \omega)$.

\begin{remark}
	For a finite dimensional space $\mathcal{H}$, the assumption of $x_n (\omega)$ to be random points 
	can be replaced by measurablity of the unstable manifold $U(\omega)$. 
	This condition is a consequence of the stable/unstable manifold theorem 
	\cite[Theorem 5.1 and 6.1]{Ruelle82} due to
	measurability of $\varphi$.
	Measurability of $U(\omega)$ is sufficient in this case since the selection theorem 
	\cite[Theorem III.9, p.67]{Castaing1977} shows that $x_0 (\omega)$ can be chosen to be  measurable
	and $\tilde{K}$ can be replaced by 
	$\left\{ y \in \mathcal{H}: \inf_{z \in K} \left\| y - z \right\| \leq 1 \right\}$.
\end{remark}

\begin{remark}
		In case of a time-invertible random dynamical system, the unstable manifold $U(\omega)$ can be
		obtained by choosing a stable manifold of the time-reversed random dynamical system. \\
		More generally, the unstable manifold can be obtained by using  \cite[Theorem 6.1]{Ruelle82}.
		Therefore, let $A(\omega)$ be an $\varphi$-invariant random point and define the cocycle
		$F^n_\omega (y)= \varphi_n (\omega, y + A(\omega) ) - A(\theta_n \omega)$.
		Under similar assumptions as in Theorem \ref{stable_manifold} 
		with $\rho (\D \omega, \D x) = \delta_{A(\omega)}(\D x) \, \mathbb{P}(\D \omega)$ 
		but supposing a positive (discrete-time) top Lyapunov exponent 
		\begin{equation*}
			\lambda (\omega) = \lim_{n \rightarrow \infty} \frac{1}{n} \log 
			\left\| D \varphi_n (\omega,A(\omega) ) \right\|
		\end{equation*}
		such that there exists $\mu >0$ with $\lambda(\omega) > \mu$ almost everywhere, we
		can apply the unstable manifold theorem \cite[Theorem 6.1]{Ruelle82}. 
		This theorem shows that there exist measurable functions $0< \alpha (\omega) < \beta (\omega)$
		such that $U (\omega)$ as in (\ref{unstable_submanifold})
		is a measurable submanifold of $\bar{B} (A(\omega), \alpha (\omega))$ almost surely. However, this does 
		not exclude the possibility that $U(\omega)$ is a single point, see Example \ref{example2}.
\end{remark}

%% file: examples.tex
\section{Examples}
\label{examples}

We provide two examples of independent iterated functions on $\mathbb{R}$. Each of them generates a random dynamical system. 
The functions will be almost surely  strictly increasing, continuous and onto and they will fix 0. 
In the first example, all trajectories which do not start at 0 converge to $\infty$ or $-\infty$ almost surely (depending on the sign 
of the initial condition) in spite of the fact the Lyapunov exponent associated to the equilibrium 0 is strictly negative. 
In particular, there is no synchronization. 
The second example just consists of an iteration of the 
inverses of the functions in the first example (in particular it is also order preserving). 
In this case the Lyapunov exponent is the negative of the one in the first example and hence strictly positive. 
From the results about the first example, we immediately obtain that the second example exhibits synchronization, 
i.e.~every compact subset of $\mathbb{R}$ contracts to 0 in probability as $n \to \infty$. Since the convergence in the first example is not only 
in probability but even almost sure we obtain, that in the second example $\{0\}$ is not only a weak attractor but even a pullback attractor (see \cite[Proposition 4.6]{CDS09}). 

We will comment on the relation of these examples to the results in the previous section after presenting the examples.


\begin{example}
	\label{example1}
	Let $ \left( \xi_n \right)_{n \in \mathbb{N}} >0 $ 
	be independent identically distributed real-valued random variables such that 
	$\mathbb{P} \left( \xi_1 \leq 2^{-k} \right) = 1 / (k-1)$ for all $k \geq 2$ and $ k \in 
	\mathbb{N}$.	
        Define the function $g:{\mathbb{R}}\times (0,\infty) \to \mathbb{R}$ by 
	\begin{equation*}
		g( z , \xi) = 
		\begin{cases}
			z/2, 								& |z| \leq 2 \xi \\
			z/ \xi + \xi -2,					& z > 2 \xi \\
			z/ \xi - \xi +2,					& z< - 2 \xi \; .
		\end{cases}
	\end{equation*}
        
	Obviously, $0$ is a fixed point of $g(\cdot,\xi)$ for each $\xi>0$ and hence $\rho:=\mathbb{P} \otimes \delta_0$ is 
	an invariant measure of the associated discrete time random 
    dynamical system $\varphi$ given by $\varphi_n (\omega, z) = g\left(\varphi_{n-1} (\omega,z), \xi_n \right)$
    for $z \in \mathbb{R}$ and $n \in \mathbb{N}$
    with state space $\mathcal{H}=\mathbb{R}$. Clearly, the Lyapunov exponent associated to $\rho$ is 
        $\log (1/2) <0 $. We write $Z_n(\omega):=\varphi_n(\omega,z)$ whenever the initial condition $Z_0=z$ is clear from the context. We will show that 
	$|Z_n|$ converges to infinity $\mathbb{P}$-almost surely whenever $Z_0 \neq 0$. 
	To see this, observe that the following properties hold for every $m \in \mathbb{N}$: 
        \begin{itemize}
          \item   $|Z_{m-1}| \geq 1$ implies $ | Z_{m} | \geq 4 |Z_{m-1} | -2  \geq 2 |Z_{m-1} |$,   
	  \item   $|Z_{m}| < 1$  implies $|Z_{m-1}|\le 4 \xi_m$.
        \end{itemize}         
%
        Assume that $ |Z_0| > 2^{-k}$ for some $k \in \mathbb{N}$. Then, 
	$ |Z_m| > 2^{-k-m}$ for all $m \in \mathbb{N}$ and therefore
	\begin{eqnarray*}
		\mathbb{P} \left( |Z_n| <1 \right)
		& \leq & \mathbb{P} \left( \xi_m > 2^{-k-m-1} \textrm{ for all } 1 \leq m \leq n \right) \\
		& = & \prod_{m=1}^n \frac{k+m-1}{k+m} = \frac{k}{k+n} 
		\quad \underrightarrow{n \rightarrow \infty} \quad 0 \; .
	\end{eqnarray*}
        Using the first of the two observations above, we obtain $|Z_n|\to \infty$ almost surely whenever $Z_0\neq 0$.  
\end{example}

\begin{example}
	\label{example2}
        Define the sequence $ \left( \xi_n \right)_{n \in \mathbb{N}} >0 $
        as above and define $f:{\mathbb{R}}\times (0,\infty) \to \mathbb{R}$ by 
	\begin{equation*}
		f( \cdot , \xi) = g^{-1}( \cdot,\xi) 
        \end{equation*}
        for each fixed $\xi>0$. As mentioned at the beginning of the section, 
        the associated random dynamical system exhibits synchronization in spite of the fact 
        that the top Lyapunov exponenent associated to its invariant measure $\rho=\mathbb{P}\otimes \delta_0$ is strictly positive. 
\end{example}

\begin{remark}
	In none of the two examples above the random dynamical system is continuously differentiable in the initial state $z$.
	This can easily be mended. Just replace the function $g$ by a function $\tilde g$ 
        which is smooth and strictly increasing in its first argument such that $|\tilde g(x)|\ge |g(x)|$ for all $x \in \mathbb{R}$ 
        and such that $\tilde g(x)=g(x)$ whenever $|x| \notin [\xi,3\xi]$. Then, the absolute values of the modified trajectories converge to $\infty$ 
        even faster than for $g$ and in Example  \ref{example2} the speed of synchronization is even faster after the modification. Note that 
        the change from $g$ to $\tilde g$ does not change the Lyapunov exponents. \\

Let us comment on the relation of the examples to the results in the previous section. Obviously, the random dynamical system $\varphi$ in 
Example \ref{example1} does not only fail to synchronize but  even fails to be asymptotically stable as defined in 
Corollary \ref{coro} (note that in this case asymptotic stability is necessary but not sufficient for synchronization by \cite{Flandoli2016}).
Therefore, the assumptions of Proposition \ref{stable_manifold} cannot hold for this example. Indeed, property (\ref{condtion2}) 
fails to hold since
\begin{eqnarray*}
		\mathbb{E} \left[ \log^+ \left\| \varphi_1 \right\|_{C^1([-1,1])} \right]
		 \geq  \mathbb{E} \left[ \log^+  \frac{1}{\xi_1} \right]  = \infty \; .
	\end{eqnarray*}

The first integrability assumption in Proposition  \ref{stable_manifold}  and negativity of the Lyapunov exponent both hold in Example 
\ref{example1} showing that (\ref{condtion2}) cannot be dropped in Proposition  \ref{stable_manifold}.   
\\
Actually, the stable manifold of Example \ref{example1} is even $\left\{ 0 \right\}$. 
Since the stable manifold of Example \ref{example1} and the unstable manifold of Example \ref{example2} coincide,
Example \ref{example2} does not satisfy the assumptions of Proposition \ref{proposition2}. 
In particular, positivity of the top Lyapunov exponent implies neither non-triviality of the unstable manifold 
nor lack of synchronization.

\end{remark}

%% file: author.bbl
\begin{thebibliography}{10}
\providecommand{\url}[1]{{#1}}
\providecommand{\urlprefix}{URL }
\expandafter\ifx\csname urlstyle\endcsname\relax
  \providecommand{\doi}[1]{DOI~\discretionary{}{}{}#1}\else
  \providecommand{\doi}{DOI~\discretionary{}{}{}\begingroup
  \urlstyle{rm}\Url}\fi

\bibitem{Arnold1988}
Arnold, L.: Random Dynamical Systems.
\newblock Springer Monographs in Mathematics. Springer-Verlag, Berlin (1998)

\bibitem{Beyn2011}
Beyn, W.J., Gess, B., Lescot, P., R\"ockner, M.: The global random attractor
  for a class of stochastic porous media equations.
\newblock Comm. Partial Differential Equations \textbf{36}(3), 446--469 (2011)

\bibitem{Castaing1977}
Castaing, C., Valadier, M.: Convex Analysis and Measurable Multifunctions.
\newblock Lecture Notes in Mathematics, Vol. 580. Springer-Verlag, Berlin-New
  York (1977)

\bibitem{Chueshov2004}
Chueshov, I., Scheutzow, M.: On the structure of attractors and invariant
  measures for a class of monotone random systems.
\newblock Dyn. Syst. \textbf{19}(2), 127--144 (2004)

\bibitem{Cranston2016}
Cranston, M., Gess, B., Scheutzow, M.: Weak synchronization for isotropic
  flows.
\newblock Discrete Contin. Dyn. Syst. Ser. B \textbf{21}(9), 3003--3014 (2016)

\bibitem{Crauel2003}
Crauel, H.: Random Probability Measures on Polish Spaces.
\newblock Stochastics Monographs. CRC Press (2003)

\bibitem{CDS09}
Crauel, H., Dimitroff, G., Scheutzow, M.: Criteria for strong and weak random
  attractors.
\newblock J. Dynam. Differential Equations \textbf{21}(2), 233--247 (2009)

\bibitem{Crauel1994}
Crauel, H., Flandoli, F.: Attractors for random dynamical systems.
\newblock Probab. Theory Related Fields \textbf{100}(3), 365--393 (1994)

\bibitem{Flandoli2016}
Flandoli, F., Gess, B., Scheutzow, M.: {Synchronization by noise}.
\newblock Probab. Theory Related Fields  (2016).
\newblock \doi{10.1007/s00440-016-0716-2}

\bibitem{Gess2016}
Flandoli, F., Gess, B., Scheutzow, M.: {Synchronization by noise for
  order-preserving random dynamical systems}.
\newblock Ann. Probab.  (2016).
\newblock \doi{10.1214/16-AOP1088}

\bibitem{Gess2011}
Gess, B., Liu, W., R\"ockner, M.: Random attractors for a class of stochastic
  partial differential equations driven by general additive noise.
\newblock J. Differential Equations \textbf{251}(4-5), 1225--1253 (2011)

\bibitem{Ruelle82}
Ruelle, D.: Characteristic exponents and invariant manifolds in {H}ilbert
  space.
\newblock Ann. of Math. \textbf{115}(2), 243--290 (1982)

\bibitem{Vorkastner2016}
Vorkastner, I.: Noise dependent synchronization of a degenerate sde  (2016).
\newblock To appear in Stoch. Dyn., arXiv:1607.01627

\end{thebibliography}
